\documentclass{fundam}

\publyear{25}
\papernumber{1}
\volume{195}
\issue{1-4}
\theDOI{10.46298/fi.12328}

\versionForARXIV

\usepackage{lineno}
\usepackage{enumitem}
\usepackage{graphicx}	
\usepackage{amsmath}
\usepackage{amssymb}
\usepackage{xspace}
\usepackage{comment}
\usepackage{hyperref}
\usepackage{float}
\usepackage{tikz}
\usetikzlibrary{positioning, automata, backgrounds}

\usepackage{subfigure,booktabs}
\usepackage{listings}

\usepackage{graphics}
\usepackage{amsfonts}
\usepackage{latexsym}

\makeatletter
\def\widebreve{\mathpalette\wide@breve}
\def\wide@breve#1#2{\sbox\z@{$#1#2$}%
     \mathop{\vbox{\m@th\ialign{##\crcr
\kern0.08em\brevefill#1{0.8\wd\z@}\crcr\noalign{\nointerlineskip}%
                    $\hss#1#2\hss$\crcr}}}\limits}
\def\brevefill#1#2{$\m@th\sbox\tw@{$#1($}%
  \hss\resizebox{#2}{\wd\tw@}{\rotatebox[origin=c]{90}{\upshape(}}\hss$}
\makeatletter

\DeclareMathOperator{\spec}{Spec}

\newcommand{\Lem}[1]{Lemma~\ref{#1}\xspace}

\newcommand{\Prop}[1]{Proposition~\ref{#1}\xspace}
\newcommand{\Thm}[1]{Theorem~\ref{#1}\xspace}

\newcommand{\Rmk}[1]{Remark~\ref{#1}\xspace}

\DeclareMathOperator{\Aut}{Aut}

\DeclareMathOperator{\Spec}{Spec}

\begin{document}

\title{Comer Schemes, Relation Algebras, and the Flexible Atom Conjecture
\footnote{We wish to thank the anonymous reviewers for their helpful feedback. A preliminary version of this work appeared in RAMiCS 2023 \cite{AALRamics23}.} }

\author{
Jeremy F.~Alm \\
  School of Mathematical and Statistical Sciences, Southern Illinois University
   \and 
   David Andrews \\
  Department of Mathematics, University of  Dallas
  \and 
  Michael Levet\thanks{ML was partially supported by J.~A.~Grochow's NSF award CISE-2047756 and a Summer Research Fellowship through the Department of Computer Science at the University of Colorado Boulder.} \\
  Department of Computer Science, College of Charleston}
  
\runninghead{J. F. Alm, D. Andrews, M. Levet}{Comer Schemes, Relation Algebras, and the Flexible Atom Conjecture}

\maketitle

\definecolor {processblue}{cmyk}{0.96,0,0,0}
\definecolor{processred}{rgb}{200, 0, 0}

\vspace{-8ex}

\begin{abstract}
In this paper, we consider relational structures arising from Comer's finite field construction, where the cosets need not be sum free. These Comer schemes generalize the notion of a Ramsey scheme and may be of independent interest. As an application, we give the first finite representation of $34_{65}$. This leaves $33_{65}$ as the only remaining relation algebra in the family $N_{65}$ with a flexible atom that is not known to be finitely representable. Motivated by this, we complement our upper bounds with some lower bounds. Using a SAT solver, we show that $33_{65}$ is not finitely representable on fewer than $24$ points, and that $33_{65}$ does not admit a cyclic group representation on fewer than $120$ points. We also employ a SAT solver to show that $34_{65}$ is not representable on fewer than $24$ points.

\end{abstract}





\section{Introduction}
Given a class of finite algebraic structures, it is natural to ask which members can be \textit{instantiated} or \textit{represented} over a finite set $S$, where there exist natural operations on $S$ corresponding to the operations of the algebraic structure. In the setting of finite groups, the representation question is answered by Cayley's theorem: every finite group can be instantiated as a finite permutation group. For this paper, we consider the class of finite relation algebras, which are Boolean lattices that satisfy certain equational axioms that capture the notion of relational composition (see Section~\ref{sec:Preliminaries} for a more precise formulation). There exist finite relation algebras that do not admit representations even over infinite sets-- see for instance \cite{Nonrepresentable2020} and the citations therein. It is essentially folklore that there are relation algebras that admit representations over infinite sets, but are not finitely representable. The so called \textit{point algebra} is one such example (for completeness, we include a proof in Appendix~\ref{appendix:PointAlgebra}). On the other hand, Comer \cite[Theorem~5.3]{ComerCombinatorial} showed that every finite integral relation algebra with a flexible atom (i.e., an atom that is as broadly compatible with non-zero products as possible-- see Section~\ref{sec:Preliminaries} for a precise definition) is representable over a countably infinite set.

It is natural to ask whether Comer's result can be strengthened to hold in the setting of finite sets. This is precisely the \textit{Flexible Atom Conjecture}, which states that every finite integral relation algebra with a flexible atom is representable over a finite set. Jipsen, Maddux \& Tuza showed that the finite symmetric integral relation algebras in which every diversity atom is flexible (denoted $\mathfrak{E}_{n+1}(1,2,3)$), are finitely representable. In particular, the algebra with $n$ flexible atoms is representable over a set of size $(2+o(1))n^2$ \cite{Jipsen}. We note that if all cycles are present, then all diversity atoms are flexible. Hence, the case considered in \cite{Jipsen} is intuitively the \textit{big end} of the Flexible Atom Conjecture.

The other extreme is when just enough cycles are present for one atom to be flexible. This case was handled by Alm, Maddux \& Manske \cite{AMM}, who exhibited a representation of the algebra $A_{n}$ obtained from splitting the non-flexible diversity atom of $6_{7}$ into $n$ symmetric atoms.\footnote{Here, we use Maddux's \cite{Madd} numbering for relation algebras such as $6_{7}$ and $32_{65}$. Note that $6_{7}$ is the relation algebra with two symmetric atoms $a, b$, where $a$ is flexible and the sole forbidden diversity cycle is $bbb$. We refer to Table~\ref{tab:cycles} for $32_{65}$.} In particular, this construction yielded a representation of $A_{2} = 32_{65}$ over a set of size $416,714,805,91$. Dodd \& Hirsch \cite{DH} subsequently improved the upper bound of the minimum representation size of $32_{65}$ to $63,432,274,896$. This was subsequently improved to 8192 by J.F. Alm \& D. Sexton (unpublished), and later $3432$ by \cite{AlmAndrews}. Finally, in \cite{alm2021improved}, the authors exhibited a representation over a set of size $1024$ for $32_{65}$, as well as the first polynomial upper bound on $\min(\spec(A_{n}))$ (where $\spec(A_{n})$ is the spectrum taken over all square representations).

In the quest for finite representations, it is desirable to constrain the search space. Furthermore, the minimum number of points required to represent a relation algebra serves as a measure of combinatorial complexity. This simultaneously motivates the study of small representations, the study of highly symmetric representations such as over groups, as well as lower bounds. There are few lower bounds in the literature. Jipsen, Maddux, \& Tuza exhibited a lower bound of $n^{2} + n + 1$ for the relation algebra $\mathfrak{E}_{n+1}(1,2,3)$ \cite{Jipsen}. In \cite{alm2021improved}, the authors showed that any representation of $A_{n}$ requires at least $2n^{2} + 4n + 1$ points, which is asymptotically double the trivial lower bound of $n^{2} + 2n + 3$. The key technique involved analyzing the combinatorial substructure induced  by the flexible atom. In the special case of $A_{2} = 32_{65}$, the authors used a SAT solver to further improve the lower bound. Namely, they showed that $32_{65}$ is not representable on a set of fewer than $26$ points.

Using similar techniques, Alm \& Levet \cite{AlmLevet2} showed that $31_{37}$ is not representable using fewer than $16$ points and $35_{37}$ is not representable using fewer than $14$ points. In the process, Alm \& Levet produced two generalizations of $31_{37}, 33_{37}$, and $35_{37}$, obtaining lower bounds against these generalizations.

The notion of a Ramsey scheme was first introduced (but not so named) by Comer \cite{ComerMonochrome}, where he used them to obtain finite representations of relation algebras. Kowalski \cite{KowalskiRamsey} later introduced the term \textit{Ramsey relation algebra} to refer to the (abstract) relation algebras obtained from embeddings into Ramsey schemes. 

In \cite[Problem~2.7]{MadduxVarieties}, Maddux raised the question as to whether Ramsey schemes exist for all number of colors. Comer communicated this problem to Trotter; and in the mid-80's, Erd\H{o}s, Szemer\'edi, \& Trotter gave a purported proof on the existence of Ramsey schemes for sufficiently many colors. Trotter communicated this proof to Comer via email, who in turn communicated it to Maddux \cite{MadduxPersonalCommunication}. Unfortunately, the construction provided by Erd\H{o}s, Szemer\'edi, \& Trotter was not correct, as it did not satisfy the mandatory cycle condition. It remains open as to whether there even exist Ramsey schemes for infinitely many values of $n$. Similarly, it remains open as to whether all of the \textit{small} (those with four or fewer atoms) integral symmetric relation algebras with at least one flexible atom outlined in \cite{Madd} admit finite representations. As of 2009 \cite{MadduxSlides}, $33_{65}, 34_{65}$, and $59_{65}$ were the only symmetric $4$-atom integral relation algebras with a flexible atom not known to admit finite representations. As of 2013 \cite{MadduxPersonalCommunication2}, $33_{37}$ and $35_{37}$ were the only non-symmetric $4$-atom integral relation algebras with a flexible atom not known to admit finite representations. In 2017, Alm \& Maddux \cite{AlmMaddux5965} gave a finite representation for $59_{65}$. In \cite{AlmLevet2}, Alm \& Levet gave finite representations for $33_{37}$ and $35_{37}$. In this paper, we give a finite representation for $34_{65}$, which leaves $33_{65}$ as the only remaining $4$-atom integral relation algebra with a flexible atom from \cite{Madd} that is not known to admit a finite representation.

There has been considerable progress in constructing Ramsey schemes for a finite number of colors using a technique due to Comer \cite{ComerMonochrome}. Intuitively, Comer's method takes a positive integer $m$ and considers finite fields $\mathbb{F}_{q}$, where $q \equiv 1 \pmod{2m}$. We next consider the unique multiplicative subgroup $H \leq \mathbb{F}_{q}^{\times}$ of order $(q-1)/m$, and check whether the cosets of $\mathbb{F}_{q}^{\times}/H$ yield a group representation (namely, taking the cosets to be the  atoms of the relation algebra-- see Definition~\ref{def:GroupRepresentation}). With the sole exception of an alternate construction of the $3$-color algebra using $(\mathbb{Z}/4\mathbb{Z})^{2}$ \cite{WHITEHEAD1975399}, all known constructions have been due to the guess-and-check finite field method of Comer \cite{ComerMonochrome}. The $m$-color Ramsey number 
\[
R_{m}(3) = R(\underbrace{3, \ldots, 3}_{m})
\]
provides an upper bound on $q$, restricting the search space for the finite fields to be considered.

Using this method, Comer \cite{ComerMonochrome} produced Ramsey schemes for $m = 2, 3, 4, 5$ colors. Using a computer, Maddux \cite{MadduxAMS2011} extended this work for $m = 6, 7$. Independently, Kowalski \cite{KowalskiRamsey} constructed Ramsey schemes over prime fields for $2 \leq m \leq 120$ colors  (with the exception of $m = 8, 13$), and Alm \& Manske \cite{AlmManskeRamsey} constructed Ramsey schemes over prime fields for $2 \leq m \leq 400$ colors, (again with the exception of $m = 8, 13$). Kowalski \cite{KowalskiRamsey} also considered non-prime fields. Alm \cite{Alm2017401AB} subsequently produced Ramsey schemes for $2 \leq m \leq 2000$ colors (excluding $m = 8, 13$) using the fast algorithm of \cite{AlmYlv}. He also substantially improved the upper bound on $p$ with respect to $m$, to $p < m^{4} + 5$, finally showing that no construction over prime fields exists for $m = 13$ \cite{Alm2017401AB}. Alm \& Levet \cite{AlmLevet2} further improved this bound to $p < m^{4} - (2-o(1))m^{3} + 5.$

Alm \& Levet \cite{AlmLevet2} generalized the notion of a Ramsey scheme to the directed (antisymmetric) setting. As a consequence, they gave finite representations for the relation algebras $33_{37}, 35_{37}$, $77_{83}$, $78_{83}$, $80_{83}$, $82_{83}$, $83_{83}$, $1310_{1316}$, $1313_{1316}$, $1315_{1316}$, and $1316_{1316}$. Only $83_{83}$ and  $1316_{1316}$ were previously known to be finitely representable, by a slight generalization of \cite{Jipsen}.

Given the success of Comer's method \cite{ComerMonochrome}, we ask the following inverse question: what relation algebras admit finite constructions via Comer's method? We investigate this question here.

\noindent \\ \textbf{Main Results.} In this paper, we extend the notion of a Ramsey scheme \cite{ComerMonochrome} by relaxing the notion that the cosets need to be sum-free. We refer to such relational structures as \textit{Comer schemes}. A Comer scheme naturally generates an abstract integral symmetric relation algebra. A priori, the cycle structure of these algebras is not clear. We investigate this question here. 

As a first application, we provide the first finite representation for the relation algebra $34_{65}$. This relation algebra has four symmetric atoms $1', a, b,$ and $c$, with forbidden cycles $bbc$ and $ccb$. The atom $a$ is flexible. Hence, by \cite[Theorem~5.3]{ComerCombinatorial}, $34_{65}$ admits a representation over a countable set.

\begin{theorem} \label{thm:3465}
The relation algebra $34_{65}$ admits a representation on $p = 3697$ points.
\end{theorem}

We obtain our finite representation of $34_{65}$ by embedding it into the integral symmetric relation algebra with $24$ diversity atoms $a_{0}, \ldots, a_{23}$. The forbidden cycles are of the form $\{ a_{i}a_{i}a_{i+12} : 0 \leq i \leq 23\}$, where the indices $i, i, i+12$ are all taken modulo $24$. This relation algebra admits a finite representation via a Comer scheme. See Section~\ref{sec:ex} for full details.

We next turn to investigating the cycle structure of these algebras. We are in fact able to produce a large number of these objects for different forbidden cycle configurations.

\begin{theorem} \label{thm:data}
We have the following.
\begin{enumerate}[label=(\alph*)]
\item For $n \in \{1, \ldots, 2000\} \setminus \{8,13\}$, the integral symmetric relation algebra on $n$ diversity atoms with forbidden cycles $\{ a_{i}a_{i}a_{i} : 0 \leq i < n\}$ admits a finite representation over a prime field Comer scheme.

\item For $n \in \{5, \ldots, 14\} \cup \{16, \ldots, 33\} \cup \{35, \ldots, 500\}$, the integral symmetric relation algebra on $n$ diversity atoms with forbidden cycles $\{ a_{i}a_{i}a_{i+1} : 0 \leq i < n\}$, where the indices $i, i, i+1$ are all taken modulo $n$, admits a finite representation over a prime field Comer scheme. 

\item For $n \in \{ 5, \ldots, 500\} \setminus \{12,14,18,28,36\}$, the integral symmetric relation algebra on $n$ diversity atoms with forbidden cycles $\{ a_{i}a_{i}a_{i+2} : 0 \leq i < n\}$, where the indices $i, i, i+2$ are all taken modulo $n$, admits a finite representation over a prime field Comer scheme. 
\end{enumerate}
\end{theorem}

\begin{remark}
We note that when $n$ is odd, forbidding the cycles $\{ a_{i}a_{i}a_{i+1} : 0 \leq i < n\}$ (where all indices are taken modulo $n$) yields the same relation algebra as when we forbid the cycles $\{ a_{i}a_{i}a_{i+2} : 0 \leq i < n\}$ (where all indices are again taken modulo $n$). See \Lem{lem:iso}.
\end{remark}

\begin{remark}
We note that Ramsey schemes also have close connections with other combinatorial structures such as association schemes, coherent configurations, and permutation groups \cite{ComerMonochrome, ComerCombinatorial}. Thus, our Comer scheme constructions may be of independent interest.
\end{remark}

The relation algebra $33_{65}$ is the integral symmetric relation algebra with atoms $1', a, b, c$, with $a$ flexible and forbidden cycles $ccc, bcc, cbb$. It is the last known relation algebra in the family $N_{65}$ with a flexible atom that is not known to admit a finite representation. To this end, we establish several lower bounds. Here, we consider not only $33_{65}$, but also $34_{65}$. By analyzing the combinatorial structure of the relation algebras in tandem with a SAT solver, we obtain the following.

\begin{theorem}
We have the following.
\begin{enumerate}[label=(\alph*)]
\item Any square representation of $33_{65}$ requires at least $24$ points.

\item Any square representation of $34_{65}$ requires at least $24$ points.
\end{enumerate}
\end{theorem}

\noindent \\ Furthermore, we rule out small cyclic group representations for $33_{65}$.

\begin{theorem}
There is no group representation of $33_{65}$ over $(\mathbb{Z}/n\mathbb{Z}, +)$ for any $n \leq 120$.    
\end{theorem}

\subsection{Mandatory cycles for algebras used}

Most of the small algebras that play a central role in this paper have three symmetric diversity atoms. There are ten possible diversity cycles in this situation -- three 1-cycles, six 2-cycles, and one 3-cycle.

The following table shows the mandatory cycles for four algebras that we consider in this paper. 

\begin{table}[h!]
    \centering
    \begin{tabular}{c|ccc|cccccc|c}
         & $aaa$ & $bbb$ & $ccc$ & $abb$ & $baa$ & $acc$ & $caa$ & $bcc$ & $cbb$ & $abc$ \\
         \hline
        $32_{65}$ &  $aaa$ &  &  & $abb$ & $baa$ & $acc$ & $caa$ &  &  & $abc$ \\
        $33_{65}$ &  $aaa$ & $bbb$ &  & $abb$ & $baa$ & $acc$ & $caa$ &  &  & $abc$ \\
        $34_{65}$ &  $aaa$ & $bbb$ & $ccc$ & $abb$ & $baa$ & $acc$ & $caa$ &  &  & $abc$ \\
        $42_{65}$ &  $aaa$ & $bbb$ & $ccc$ & $abb$ &  &  & $caa$ & $bcc$ &  & $abc$ \\
    \end{tabular}
    \caption{Mandatory cycles for four algebras}
    \label{tab:cycles}
\end{table}

\section{Preliminaries} \label{sec:Preliminaries}

\begin{definition} \label{Def:RAs}
A \textit{relation algebra} is an algebra $\langle A, \leq, \land, \vee, \neg, 0, 1, \circ, \breve{}, 1^{\prime} \rangle$ that satisfies the following.
\begin{itemize}
\item $\langle A, \leq, \land, \vee, \neg, 0, 1 \rangle$ is a Boolean algebra, with $\neg$ the unary negation operator, $0$ the identity for $\vee$, and $1$ the identity for $\land$. 

\item $\langle A, \circ, 1^{\prime} \rangle$ is a monoid, with $1^{\prime}$ the identity with respect to $\circ$. We refer to $\circ$ as \textit{(relational) composition}. So we have that relational composition is associative, and there is an identity $1^{\prime}$ with respect to $\circ$. 

\item $\breve{}$ \, is the unary converse operation, and the reduct $\langle A, \circ, \breve{}\;\rangle$ is an involuted monoid. Namely, $\breve{\breve{a}} = a$ for all $a \in A$, and $\widebreve{a \circ b} = \breve{b} \circ \breve{a}$ for all $a, b \in A$.

\item Converse and composition both distribute over disjunction. Precisely, for all $a, b, c \in A$, we have that:
\begin{align*}
&\widebreve{(a \vee b)} = \breve{a} \vee \breve{b}, \text{ and } \\
&(a \vee b) \circ c = (a \circ c) \vee (b \circ c).
\end{align*}
\item (Triangle Symmetry) For all $a, b, c \in A$, we have that:
\[
(a \circ b)\land c = 0 \iff (\breve{c} \circ a)\land \breve{b} = 0 \iff (b\circ \breve{c})\land \breve{a} = 0.
\]
\end{itemize}
\end{definition}

\noindent When the relation algebra is understood, we simply write $A$ rather than $\langle A, \leq, \land, \vee, \neg, 0, 1, \circ, \breve{}, 1^{\prime} \rangle$.

\begin{definition}
Let $A$ be a relation algebra. We say that $A$ is \textit{integral} if whenever $x \circ y = 0$, we have that $x = 0$ or $y = 0$.
\end{definition}

\begin{definition} \label{def:Atom}
Let $A$ be a relation algebra. We say that $a \in A$ is an \textit{atom} if $a \neq 0$ and $b < a \implies b = 0$. Furthermore, we say that $a$ is a \textit{diversity atom} if $a$ also satisfies $a \land 1^{\prime} = 0$. We say that a diversity atom is \textit{symmetric} if $a = \breve{a}$. 
\end{definition}

In this paper, attention will be restricted to finite relation algebras. This ensures that every element of a relation algebra can be written as the join of finitely many atoms.

\begin{remark}
In the special case when $a, b, c$ are diversity atoms, the Triangle Symmetry axiom defines an equivalence relation on such triples that corresponds to the symmetries of the triangle.
\end{remark}

\begin{definition}
For atoms $a, b, c$, the triple $(a, b, c)$ -- usually denoted $abc$ -- is called a \textit{cycle}. We say that the cycle $abc$ is \textit{forbidden} if $(a \circ b) \land c = 0$ and \textit{mandatory} if $a \circ b \geq c$. If $a, b, c$ are diversity atoms, then $abc$ is called a \textit{diversity cycle}. 
\end{definition}

Any cycle that is not forbidden is mandatory \cite[Chapter~6]{Madd}. 

\begin{remark} \label{obs:Composition}
We note that for symmetric integral relation algebras, the composition operation $\circ$ is determined by the mandatory diversity cycles. 
\end{remark}

\begin{definition} \label{def:flexible}
Let $f$ be a symmetric diversity atom. We say that $f$ is \textit{flexible} if for all diversity atoms $a, b$, we have that $abf$ is mandatory.
\end{definition}

\begin{definition}
We say that a relation algebra $A$ is \textit{representable} if there exists a set $U$ and an equivalence relation $E \subseteq U \times U$ such that $A$ embeds into
\[
\langle 2^{E}, \subseteq, \cup, \cap, ^{c}, \circ, ^{-1}, \emptyset, E, \text{Id}_{U} \rangle.
\]

\noindent Here, $^{c}$ is set complementation, and $^{-1}$ is the relational inverse. 
\end{definition}

We will only be concerned with \textit{simple} relation algebras, in which case, representability is equivalent to the existence of a set $U$ with the choice of $E = U \times U$. We call such a representation \textit{square}.

\begin{definition}
Let $A$ be a finite relation algebra. Denote:
\[
\text{Spec}(A) := \{ \alpha \leq \omega : A \text{ has a square representation over a set of cardinality } \alpha\}.
\]
\end{definition}

\noindent If $\text{Spec}(A)$ contains a natural number, then we say that $A$ admits a finite representation. The minimum element in $\spec(A)$ serves as a measure of combinatorial complexity for the relation algebra.

\begin{definition} \label{def:GroupRepresentation}
Let $G$ be a group. Define the \textit{complex algebra}:
\[
\mathrm{Cm}(G) = \langle 2^{G}, \subseteq, \cup, \cap, ^{c}, \cdot, ^{-1}, \emptyset, G, \{e\} \rangle,
\]

\noindent where $^{c}$ is the setwise complementation, $\cdot : 2^{G} \times 2^{G} \to 2^{G}$ is the map sending the pair $(X,Y)$ to 
\[
X \cdot Y = \{ xy : x \in X, y \in Y, \text{ and } xy \text{ is considered in the group} \}, 
\]

\noindent $^{-1} : 2^{G} \to 2^{G}$ maps $X \mapsto X^{-1} := \{ x^{-1} : x \in X \}$,  and $e$ is the identity in the group. A \textit{group representation} of the relation algebra $A$ over the group $G$ is an injective homomorphism $\varphi : A \to \text{Cm}(G)$. 
\end{definition}

\begin{remark}
    Given a group representation $\varphi$, we may construct a representation $A$ as follows. Take $U = G$. Now for each $X \subseteq G$ in $\text{Im}(\varphi)$, map $X$ to the relation $\{ (u, v) : uv^{-1} \in X \} \subseteq G \times G$. 
\end{remark}

\begin{remark}  \label{rmk:Constraints}
In light of Remark~\ref{obs:Composition}, deciding whether an integral relation algebra $A$ admits a finite representation is equivalent to finding a finite complete graph with an appropriate edge coloring. Note that if $A$ is symmetric, we consider the undirected complete graph. If $A$ is not symmetric, then we instead consider the directed complete graph. In this paper, attention will be restricted to the symmetric case.

Given a complete graph on $m$ vertices, we seek to color the edges using the diversity atoms as colors. Now suppose $a \in A$ is a diversity atom, and $uv$ is an edge colored $a$ (we abuse notation by using the diversity atom as the label for the edge color). If $abc$ is a mandatory diversity cycle, then there must exist a vertex $w$ such that $vw$ is colored $b$ and $uw$ is colored $c$. If instead $abc$ is a forbidden diversity cycle, then for any vertex $w$, we have that either $vw$ is not colored $b$ or $uw$ is not colored $c$. Now if there exists such a finite $m$, then $A$ is finitely representable. We may view the representation size $m$ as a measure of combinatorial complexity for a relation algebra, and so it is of interest to find the smallest such $m$.
\end{remark}

\subsection{Ramsey Schemes}

\begin{definition} \label{def:RamseyScheme}
Let $U$ be a set, and $m \in \mathbb{Z}^{+}$. A \textit{Ramsey scheme} in $m$ colors is a partition of $U \times U$ into $m+1$ pairwise-disjoint binary relations $\text{Id}, R_{0}, \ldots, R_{m-1}$ such that the following conditions hold:
\begin{enumerate}[label=(\Alph*)]
\item $R_{i}^{-1} = R_{i}$,
\item $R_{i} \circ R_{i} = R_{i}^{c}$, and
\item For all pairs of distinct $i, j \in [m-1]$, $R_{i} \circ R_{j} = \text{Id}^{c}$.
\end{enumerate}

\noindent Here, $\text{Id} = \{ (u, u) : u \in U\}$ is the identity relation over $U$.
\end{definition}

The usual method of constructing the relations $R_{0}, \ldots, R_{m-1}$ is a \textit{guess-and-check} approach due to Comer \cite{ComerMonochrome}, which works as follows. Fix $m \in \mathbb{Z}^{+}$.  We search over primes $p \equiv 1 \pmod{2m}$, where a desirable $p$ satisfies the following. Let $X_{0} := H \leq \mathbb{F}_{p}^{\times}$ be the unique subgroup of order $(p-1)/m$. Now let $X_{1}, \ldots, X_{m-1}$ be the cosets of $\mathbb{F}_{p}^{\times}/X_{0}$. In particular, as $\mathbb{F}_{p}^{\times}$ is cyclic, we may write $X_{i} = g^{i}X_{0} = \{ g^{am+i} : a \in \mathbb{Z}^{+} \},$ where $g$ is a generator of $\mathbb{F}_{p}^{\times}$. Suppose that $X_{0}, \ldots, X_{m-1}$ satisfy the following conditions: 
\begin{enumerate}[label=(\alph*)]
\item $X_{i} = -X_{i}$, for all $0 \leq i \leq m-1$,
\item $X_{i} + X_{i} = \mathbb{F}_{p} \setminus X_{i}$, for all $0 \leq i \leq m-1$, and
\item For all distinct $0 \leq i, j \leq m-1$, $X_{i} + X_{j} = \mathbb{F}_{p} \setminus \{0\}$.
\end{enumerate}
\noindent For each $0 \leq i \leq m-1$, define $R_{i} := \{(x,y) \in \mathbb{F}_{p} \times \mathbb{F}_{p} : x - y \in X_{i} \}.$ Here, the sets $R_{0}, \ldots, R_{m-1}$ together with $\text{Id} = \{ (u, u) : u \in \mathbb{F}_{p} \}$ are the atoms in our relation algebra. It is easy to check that conditions (a)--(c) on the sets $X_{0}, \ldots, X_{m-1}$ imply that conditions (A)--(C) from the definition of a Ramsey scheme are satisfied for the relations $R_{0}, \ldots, R_{m-1}$. 

We note that condition (b), that $X_{i} + X_{i} = \mathbb{F}_{p} \setminus X_{i}$, indicates that each $X_{i}$ is sum-free. The fact that $p \equiv 1 \pmod{2m}$ implies that $X_{0}$ has even order. It follows that $X_{0}$ is symmetric; i.e., $X_{0} = -X_{0}$. 

In \cite{Alm2017401AB}, Ramsey schemes were constructed for all $m \leq 2000$ except for $m = 8, 13$, and it was shown that if $p > m^{4} + 5$, then $X_{0}$ contains a solution to $x + y = z$ and so fails condition (b) of Comer's method. So in such cases, Comer's construction fails to yield an $m$-color Ramsey  scheme.

\section{Integral Symmetric Relation Algebras with Forbidden Cycle Configurations}

In this section, we generalize the notion of a Ramsey scheme to the setting where the cosets need not be sum-free. We then examine the relation algebras generated by these schemes. Precisely, we drop condition (b) in Comer's method, which corresponds to dropping condition (B) in Definition~\ref{def:ComerSchemes}.

\begin{definition} \label{def:ComerSchemes}
Let $p$ be a prime, $U = \mathbb{F}_{p}$, and $m$ a divisor of $(p-1)/2$. Let $H \leq \mathbb{F}_{p}^{\times}$ be the unique subgroup of index $m$, and let $X_{0}, \ldots, X_{m-1}$ be the cosets of $\mathbb{F}_{p}^{\times}/H$. A \textit{Comer scheme} in $m$ colors is a partition of $U \times U$ into $m+1$ binary relations $\text{Id}, R_{0}, \ldots, R_{m-1}$ as follows. Let $\text{Id}= \{ (x,x) : x \in \mathbb{F}_{p}\}$ and $R_i = \{ (x,y) \in \mathbb{F}_{p} \times \mathbb{F}_{p} : x - y \in X_{i} \}$. 
\end{definition}

\begin{remark}
We will be particularly interested in Comer schemes that satisfy the following additional condition: for all distinct $i, j, k \in [m-1]$, 
\begin{align}
R_{i} \circ R_{j} \supseteq R_{k}. \tag{$B^{\prime}$}    
\end{align}
\end{remark}

We construct Comer schemes using Comer's method \cite{ComerMonochrome}. Fix $m \in \mathbb{Z}^{+}$, and let $p \equiv 1 \pmod{2m}$ be a prime. Let $X_{0} := H \leq \mathbb{F}_{p}^{\times}$ be a subgroup of order $(p-1)/m$. Now let $X_{1}, \ldots, X_{m-1}$ be the cosets of $\mathbb{F}_{p}^{\times}/X_{0}$. In particular, as $\mathbb{F}_{p}^{\times}$ is cyclic, we may write $X_{i} = g^{i}X_{0} = \{ g^{am+i} : a \in \mathbb{Z}^{+} \},$ where $g$ is a generator of $\mathbb{F}_{p}^{\times}$. Suppose that $X_{0}, \ldots, X_{m-1}$ satisfy the following conditions:
\begin{enumerate}[label=(\alph*)]
\item $X_{i} = -X_{i}$, for all $0 \leq i \leq m-1$,

\item For all distinct $0 \leq i, j \leq m-1$, $X_{i} + X_{j} = \mathbb{F}_{p} \setminus \{0\}$.
\end{enumerate}
\noindent  For each $0 \leq i \leq m-1$, define $R_{i} := \{(x,y) \in \mathbb{F}_{p} \times \mathbb{F}_{p} : x - y \in X_{i} \}.$ Here, the sets $R_{0}, \ldots, R_{m-1}$ are the atoms in our relation algebra. It is easy to check that conditions (a) and (b) on the sets $X_{0}, \ldots, X_{m-1}$ imply that conditions (A) and (C) from the definition of a Ramsey scheme are satisfied for the relations $R_{0}, \ldots, R_{m-1}$. That is, we drop condition (B) from the definition of a Ramsey scheme.


We collect some preliminary observations about these Comer schemes. First, we observe that Comer's construction yields relation algebras that have \textit{rotational symmetry}.

\begin{lemma}\label{lem:1}
 Let  $n \in \mathbb{Z}^{+}$ and let
 $p = nk + 1$ be a prime number and $g$ a primitive root modulo $p$. 
 
 For $i \in \left\{0,1,\ldots,n-1\right\}$, define
\[ X_{i} = \left\{g^{i},g^{n + i},g^{2n + i},\ldots,g^{(k-1)n + i}\right\}\subseteq \mathbb{Z}/p\mathbb{Z}.\]
Then  $(X_0 + X_j) \cap X_{\ell} = \emptyset$ if and only if  $(X_i + X_{i+j}) \cap X_{i+\ell} = \emptyset$.

 \end{lemma}

\begin{shortproof}
Multiply through by $g^i$.
\end{shortproof}

We now formalize the notion of an automorphism that respects this rotational symmetry.

\begin{definition}
Let $p=nk+1$ be prime with $k$ even, and let the $X_i$'s be as in Lemma \ref{lem:1}. Let $C(p,n)$ denote the proper relation algebra generated by the sets $R_{i}$ as in the definition of a Comer scheme. Then automorphisms of $C(p, n)$ can be seen as permutations of the atoms, as follows.
\[
    \Aut(C(p,n)) =\{\pi\in S_n: R_i\circ R_j\supseteq R_k \Leftrightarrow R_{\pi(i)}\circ R_{\pi(j)}\supseteq R_{\pi(k)}\}.
\]

\end{definition}

\noindent
Of course, the condition on the relations $R_i$ is equivalent to the following condition on the $X_i$'s:
\[
X_i+X_j\supseteq X_k\Leftrightarrow X_{\pi(i)}+X_{\pi(j)}\supseteq X_{\pi(k)}.
\]


It is natural to ask about relation algebras with forbidden cycles. In particular, given a forbidden cycle scheme and a prescribed number of atoms $n$, can we realize the corresponding relation algebra with a Comer scheme? This motivates the following definition.

\begin{definition}
Let $\mathcal{A}_n([i,i+j,i+\ell])$ denote the integral RA with $n$ symmetric diversity atoms $a_0, \ldots, a_{n-1}$ whose forbidden cycles are those of the form $\{ a_i a_{i+j} a_{i+\ell} : 0 \leq i < n \}$, with indices considered modulo $n$.
\end{definition}

\begin{remark}
We note that for any $m$-color Comer scheme, the automorphism group has a copy of $\mathbb{Z}/m\mathbb{Z}$. Thus, if we forbid one cycle --- say $a_{0}a_{0+j}a_{0+\ell}$ --- then we forbid $a_{i}a_{i+j}a_{i+\ell}$ for all $i$. See \Lem{lem:1}.
\end{remark}

\noindent We will be particularly interested in $\mathcal{A}_n([i,i,i+j])$, as these forbid a rotational class of $2$-cycles (bichromatic triangles). Furthermore, setting $j = 0$ yields  Ramsey schemes.

\begin{lemma}
$\Aut(\mathcal{A}_n([i,i,i+1])) \cong \mathbb{Z}/n\mathbb{Z}$.
\end{lemma}

\begin{proof}
 Suppose $\pi \in \Aut(\mathcal{A}_n([i,i,i+1]))$ and $\pi(0) = x$. Since $a_0a_0a_1$ is forbidden, $a_{\pi(0)}a_{\pi(0)}a_{\pi(1)}$ is forbidden as well, so $a_xa_xa_{\pi(1)}$ is forbidden. But this forces $\pi(1)=x+1 \pmod{n}$, since if $xxy$ is forbidden, $y$ must be $x+1$. Similarly, $\pi(2)$ must be $x+2 \pmod{n}$, and so on. So $\pi$ must take the form $\pi(s) = x+s \pmod{n}$. All such permutations are clearly in $\Aut(\mathcal{A}_n([i,i,i+1]))$, and we have shown they are the only ones. So $\Aut(\mathcal{A}_n([i,i,i+1])) \cong \mathbb{Z}/n\mathbb{Z}$.
\end{proof}

\begin{lemma}\label{lem:iso}
If $\gcd(j,n)=1$, then $\mathcal{A}_n([i,i,i+j])) \cong \mathcal{A}_n([i,i,i+1])$.
\end{lemma}

\begin{proof}
 Let $\rho : \mathrm{At}(\mathcal{A}_n([i,i,i+1])) \to \mathrm{At}(\mathcal{A}_n([i,i,i+j]))$ be given by $a_i \mapsto a_{j\cdot i \pmod{n} }$. Since $\gcd(j,n)=1$, $\rho$ is a bijection. It is easy to check that $\rho$ preserves the forbidden cycles.
\end{proof}

The next lemma tells us that for Comer algebras, the isomorphism in Lemma \ref{lem:iso} arises in a particularly nice way.

\begin{lemma}
Let $[X_{i}, X_{j}, X_{\ell}]$ denote the set of cycles of the form $a_{i}a_{j}a_{\ell}$, where the indices are taken modulo $n$. If $C(p,n)$ has forbidden cycles $[X_i, X_i, X_{i+j}]$  and $\gcd(j,n)=1$, then $X_j$ contains a primitive root $g$, and reindexing using $g$ as a generator will give forbidden cycles $[X_i, X_i, X_{i+1}]$.
\end{lemma}

\begin{proof}
 Let $g$ be the primitive root that gives the indexing with forbidden cycles $[X_i, X_i, X_{i+j}]$. Now $g^\ell$ is also a primitive root modulo $p$ if $\gcd(\ell, p-1)=1$. We want $g^\ell\in X_j$, so we want to find an integer $a$ with $\gcd(an+j, p-1)=1$. Since $\gcd(j,n)=1$, Dirichlet's theorem on primes in arithmetic progressions gives some prime $p'=an+j$, and clearly $\gcd(p',p-1)=1$. Then $g^{p'}$ is a primitive root and is in $X_j$. 
\end{proof}

\begin{lemma}
If $\gcd(j,n)>1$, then $\Aut(\mathcal{A}_n([i,i,i+j]))$ contains a non-identity permutation $\pi$ that has fixed points. Hence $\mathcal{A}_n([i,i,i+j]) \not\cong \mathcal{A}_n([i,i,i+1])$.
\end{lemma}

\begin{proof}
 Let $x=\gcd(j,n)>1$. Consider the permutation $\pi = (0\ x\  2x\  3x  \ldots )$, written in cycle notation. We claim that $\pi \in \Aut (\mathcal{A}_n([i,i,i+j])))$.  Consider the forbidden cycle $a_0a_0a_j$. Write $j=bx$ for some positive integer $b$.  Under $\pi$, this cycle $a_0a_0a_{bx}$ gets mapped to $a_xa_xa_{(b+1)x}$, and since $(b+1)x = x+j$, the cycle $a_xa_xa_{(b+1)x}$ is forbidden. In fact, $\pi$ just permutes the forbidden cycles $a_{\ell x}a_{\ell x}a_{\ell x + j}$ and leaves the other forbidden cycles fixed.
\end{proof}

\begin{example}
Consider $\mathcal{A}_6([i,i,i+2])$. Then $j=x=2$. The permutation $(0\ 2\ 4)$ permutes the forbidden cycles $a_0a_0a_2$, $a_2a_2a_4$, and $a_4a_4a_0$.  See Figure \ref{fig:triangles}.
\end{example}

\begin{figure}[t]
    \centering
    \includegraphics[width=3.5in]{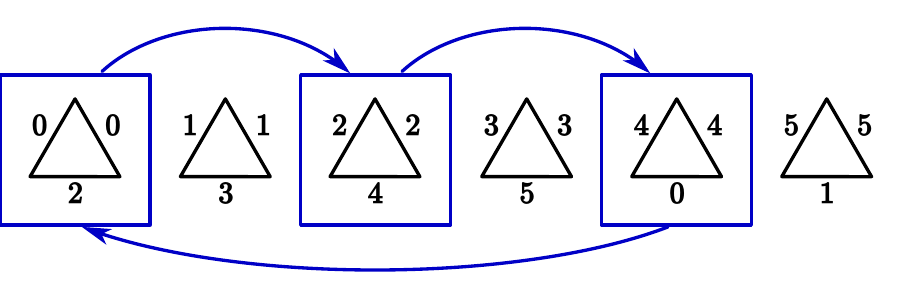}
    \caption{Depiction of the action of the permutation $(0\ 2\ 4)$ on the forbidden cycles of $\mathcal{A}_6([i,i,i+2])$.}
    \label{fig:triangles}
\end{figure}

The following lemma from Alon and Bourgain gives us just what we need to show that if $p$ is large relative to $n$, then $C(p,n)$ has only flexible diversity atoms.

\begin{lemma}[\cite{AlonBourgain}, Proposition 1.4]\label{lem:alon}
Let $q$ be a prime power and let $A$ be a multiplicative subgroup of
the finite field $\mathbb{F}_q$ of size $|A| = d \geq q^{1/2}$. Then, for any two subsets $B$, $C \subset \mathbb{F}_q$ satisfying $|B| \cdot |C| \geq q^3 / d^2$, there are $x \in B$ and $y \in C$ so that $x + y \in A$. 
\end{lemma}

\begin{lemma}\label{lem:flex}
If $p>n^4+5$, then every diversity atom of $C(p,n)$ is flexible.
\end{lemma}
\begin{proof}
 We need to show $(X_i+X_j)\cap X_0\neq \varnothing$ for arbitrary $i$ and $j$. Set $q=p$, $A=X_0$, $B=X_i$, and $C=X_j$ in Lemma \ref{lem:alon}. Then $|A|=|B|=|C|=(p-1)/n$. Then we need $|B||C| \geq q^3 / d^2$, which translates to $(p-1)^4 \geq n^4p^3$, which is satisfied when $p>n^4+5$. Then all diversity cycles are mandatory by Lemma~\ref{lem:1}.
\end{proof}

\section{Constructing Comer Schemes}

In this section, we document our constructions for \Thm{thm:data}. Some data are summarized in Table \ref{tab:my_label} below. While for some small $n$, there is no construction of a Comer RA representation for $\mathcal{A}_n([i,i,i+j])$ for $j=0,1,2$, it would seem for large enough $n$ there is always some modulus $p$ that works. 

Representations of $\mathcal{A}_n([i,i,i+1])$ exist for all $35\leq n \leq 500$. In Figure \ref{fig:12cycle}, we compare the smallest modulus $p$ for representations over $C(p,n)$ for $\mathcal{A}_n([i,i,i])$ vs $\mathcal{A}_n([i,i,i+1])$. The growth is a bit slower for the latter. 

\begin{table}[H]\label{tab}
    \centering
    \begin{tabular}{cccc}
    n & $\mathcal{A}_n([i,i,i]))$ & $\mathcal{A}_n([i,i,i+1]))$ & $\mathcal{A}_n([i,i,i+2]))$ \\
    \toprule
    1 & 2  & x  & x  \\[-0.4ex]
    2 & 5  & x  & x  \\[-0.4ex]
    3 & 13  & x  & x  \\[-0.4ex]
    4 & 41  &  x & x  \\[-0.4ex]
    5 & 71  & 61  & --  \\[-0.4ex]
    \cmidrule{2-4}
    6 &  97 & 109  & x  \\[-0.4ex]  
    7 & 491  & 127  & --  \\[-0.4ex]
    8 & x  & 257  & x  \\[-0.4ex]
    9 & 523 & 307  &  -- \\[-0.4ex]
    10 & 1181  & 641  & 421  \\[-0.4ex]
    \cmidrule{2-4}
    11 & 947  & 331  & --  \\[-0.4ex]
    12 & 769  & 673  & x  \\[-0.4ex]
    13 & x  & 667  & --  \\[-0.4ex]
    14 &  1709 & 953  & x  \\[-0.4ex]
    15 & 1291  & x  & x  \\[-0.4ex]
    16 &  1217 & 2593  & 1697  
\end{tabular}
    \caption{Smallest modulus for a representation over a $C(p,n)$, or x if none exists.  The ``--'' indicates that an entry is redundant (in light of Lemma \ref{lem:iso}).}
    \label{tab:my_label}
\end{table}

\vspace{-5ex}

\begin{figure}[H]
    \centering
    \includegraphics[width=5in]{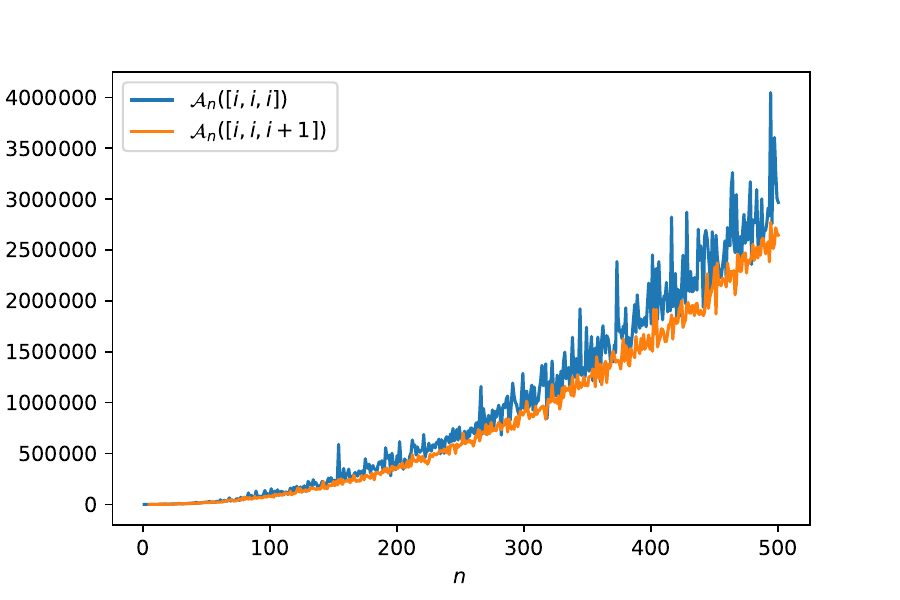}
    \caption{Smallest modulus $p$ over which $\mathcal{A}_n([i,i,i])$ and $\mathcal{A}_n([i,i,i+1])$ are representable as a $C(p,n)$}
    \label{fig:12cycle}
\end{figure}

\section{A cyclic group representation of relation algebra $34_{65}$}\label{sec:ex}

As an application, we give the first known finite representation of $34_{65}$.  Relation algebra $34_{65}$ has four symmetric atoms $1'$, $a$, $b$, and $c$, with forbidden cycles $bbc$ and $ccb$.  The atom $a$ is flexible, hence $34_{65}$ is representable over a countable set. 

We noticed that it would be sufficient to find a prime $p=nk+1$, $k$ and $n$ both even, such that $C(p,n)$ has $[i, i, i+n/2]$ as its only forbidden class.  Then we could map $b$ to $X_0$, map $c$ to $X_{n/2}$, and map $a$ to the union of all the other $X_i$'s; in other words, $34_{65}$ embeds in $\mathcal{A}_n([i,i,i+n/2]))$ for all even $n>4$. There's no limit to how big $p$ can be, since $n$ can also be as large as necessary; we just throw ``everything else'' into the image of $a$. A computer search using the fast algorithm from \cite{AlmYlv} quickly found a hit: for $p=3697$ and $n=24$, $[i, i, i+12]$ alone is forbidden. (This indexing is for the primitive root $g=5$ in $\mathbb{F}_{3697}$.)

\section{Lower Bounds on Arbitrary Representations}

\subsection{$\mathcal{A}_{n}([i, i+j, i+\ell])$}

\begin{proposition} \label{prop:GeneralLowerBound}
Let $i, j, \ell \in \{0, \ldots, n-1\}$. Any finite representation of $\mathcal{A}_{n}([i,i+j,i+\ell])$ requires at least $n^{2}$ points. 
If furthermore $j = \ell$, then at least $n^{2}+1$ points are required.
\end{proposition}

\begin{proof}
We phrase the proof in the language of graph theory (see \Rmk{rmk:Constraints}). The atoms constitute the edge colors of a complete graph. Fix $i \in [n]$, and let $uv$ be an edge colored $a_{i}$ (where here, $a_{i}$ is a diversity atom of $\mathcal{A}_{n}([i, i+j, i+\ell])$). 

Any non-forbidden triangle must be included in our graph. Let $c_{1}, c_{2}, c_{3} \in \mathcal{A}_{n}([i, i+j, i+\ell])$ be diversity atoms (colors on the graph). Suppose that $c_{1}c_{2}c_{3}$ is a mandatory triangle. If $xy$ is an edge colored $c_{1}$ and $w$ is a vertex such that $xw$ is colored $c_{2}$ and $yw$ is colored $c_{3}$, then we say that $w$ witnesses the $(c_{2}, c_{3})$ need for $xy$.

If $j \neq \ell$, then $(a_{i+j}, a_{i+\ell})$ and $(a_{i+\ell}, a_{i+j})$ are forbidden. So $uv$ has $n^{2} - 2$ needs, for which we need distinct witnesses. Together with $u,v$, this yields $n^{2}$ points in total.  

If instead $j = \ell$, we only forbid $(a_{i+j}, a_{i+j})$. Thus, $uv$ has $n^{2}-1$ needs in this case. Together with $u,v$, this yields $n^{2}+1$ points in total.
\end{proof}

\begin{remark} \label{rmk:LowerBounds}
In light of \Prop{prop:GeneralLowerBound}, we obtain that the representations we found for $\mathcal{A}_{n}([i,i,i])$ with $n = 1, 2$ were indeed minimal representations. \Prop{prop:GeneralLowerBound} only provides a lower bound of $10$ points for $\mathcal{A}_{3}([i,i,i])$. However, using a SAT solver (see \href{https://github.com/michaellevet/Comer-Schemes/blob/main/A3_iii.ipynb}{here}), we were able to show that there is no representation on fewer than $13$ points. Thus, the representation we found was indeed minimal. (This result seems to be folklore.) 

The algebra $\mathcal{A}_{3}([i,i,i])$, is in fact the $3$-color Ramsey algebra; that is, the relation algebra with diversity atoms $1', a, b, c$ and forbidden cycles $aaa, bbb, ccc$. It is folklore amongst relation algebra specialists that $\Spec(\mathcal{A}_{3}([i,i,i])) = \{13,16\}$, though a proof appears not to have been written down. Our result verifies that the minimum element in $\Spec(\mathcal{A}_{3}([i,i,i]))$ is indeed $13$.

In the case of $n = 4$, we used a SAT solver (see \href{https://github.com/michaellevet/Comer-Schemes/blob/main/A4_iii.ipynb}{here}) to establish a lower bound of $21$ points required to represent $\mathcal{A}_{4}([i,i,i])$. In particular, it follows that if there exists a prime field Comer scheme representation, then at least $23$ points are required.
\end{remark}

We now document our SAT solver approach. Our goal is to determine whether $\mathcal{A}_{n}([i,i,i])$ has a representation on $m$ points. In particular, we design a formula $\Phi^{(m)}$ that is a necessary condition for $\mathcal{A}_{n}([i,i,i])$ to be representable on $m$ points. So if $\Phi^{(m)}$ is unsatisfiable, then $\mathcal{A}_{n}([i,i,i])$ is not representable on $m$ points.

For $0 \leq i < j < m$ and $0 \leq k < n$, we have a variable $\phi_{i,j,k} = \texttt{true}$ which indicates that the edge $x_{i}x_{j}$ is colored $k$. Here, we interpret $k$ to be the atom $a_{k}$. We now define a formula to assert that each edge $x_{ij}$ receives exactly one color.
\[
\Phi_{0,n}^{(m)} := \left(\bigwedge_{i < j} \bigvee_{k=0}^{n-1} \phi_{i,j,k} \right) \land \left( \bigwedge_{i < j} \bigwedge_{0 \leq k < h < m} \neg \phi_{i,j,k} \vee  \neg \phi_{i,j,h} \right).
\]

Now fix an atom. Without loss of generality, we may choose $a_{0}$. Consider an edge in our graph labeled using $a_{0}$. This edge has $n^{2} - 1$ needs. For $0 \leq i < n$, let $R_{i}$ be the set of edges incident (at either endpoint of this edge labeled $a_0$) to $a_{0}$ that are colored $i$. Define:
\[
\Phi_{1,n}^{(m)} := \bigwedge_{i=0}^{n-1} \bigwedge_{(u, v) \in R_{i}} \phi_{u,v,i}.
\]

Now $\Phi_{1,n}^{(m)}$ asserts that the needs of $a_{0}$ are satisfied, and that the corresponding edges of the graph are colored correctly. 

For atom $0 \leq i < n$, the needs are:
\[
N_{i} := \{0, \ldots, n-1\}^{2} \setminus \{ (i,i)\}.
\]

Define:
\[
\Phi_{2,n}^{(m)} := \bigwedge_{i=0}^{n-1} \bigwedge_{(u,v) \in R_{i}} \bigwedge_{(c_{1}, c_{2})  \in N_{i}} \bigvee_{k \not \in \{u,v\}} (\phi_{u,k,c_{1}} \land \phi_{v,k,c_{2}}),
\]
which asserts that each pre-colored edge has its needs satisfied. We next define a formula to assert that every edge that is not pre-colored has its needs satisfied. Let:
\[
R := \bigcup_{i=0}^{n-1} R_{i}.
\]

Now define:
\[
\Phi_{3,n}^{(m)} := \bigwedge_{(u,v) \not \in R} \bigvee_{i=0}^{n-1} \left[\phi_{u,v,i} \land \left( \bigwedge_{(c_{1}, c_{2}) \in N_{i}} \bigvee_{k \not \in \{u,v\}} (\phi_{u,k,c_{1}} \land \phi_{v,k,c_{2}}) \right) \right]
\]

Finally, we define a formula to forbid monochromatic triangles.

\[
\Phi_{4,n}^{(m)} := \bigwedge_{i=1}^{n-1} \bigwedge_{0\leq u < v < w < m} \neg(\phi_{u,v,i} \land \phi_{u,w,i} \land \phi_{v,w,i}).
\]

Finally, define:
\[
\Phi_{n}^{(m)} := \bigwedge_{j=0}^{4} \Phi_{j}^{(m)}.
\]

Using a SAT Solver, we verified that $\Phi_{3}^{(m)}$ is unsatisfiable for $m = 11, 12$. Similarly, we verified that $\Phi_{4}^{(m)}$ is unsatisfiable for $m = 17, \ldots, 20.$

\subsection{$33_{65}$}

In this section, we consider the relation algebra $33_{65}$, which has atoms $1', a, b, c$, with $a$ flexible and the following cycles forbidden: $ccc, bcc, cbb$. We use the language of graph theory to discuss the relation algebra (see \Rmk{rmk:Constraints}). The atoms constitute the edge colors of a complete graph. We use the color red to correspond to the flexible atom, blue to correspond to the atom $b$, and green to correspond to the atom $c$. Thus, the forbidden triangles are precisely the ones containing all-green edges, or only blue and green edges with at least one green edge.

\begin{lemma} \label{lem:LowerBound1}
If $33_{65}$ is finitely representable, then any finite representation must have at least $15$ points.
\end{lemma}

\begin{proof}
We refer to Figure~\ref{fig:3365}. Let $x_{0}x_{1}$ be colored according to the flexible atom, for which we use the color red. The needs of $x_{0}x_{1}$ are precisely $\{r, b, g\} \times \{r, b, g\}$, yielding a total of $11$ points. Now fix $i \in \{0, 1\}$. Suppose that $x_{i}x_{j}, x_{i}x_{k}$ incident to $x_{i}$ where at least one such edge is green, and the other is either blue or green. As the all green, blue-blue-green, and blue-green-green triangles are forbidden, $x_{j}x_{k}$ is necessarily red. In particular, this implies that the following edges are red (see Figure~\ref{fig:3365}):
\begin{itemize}
\item $x_{2}: x_{2}x_{3}, x_{2}x_{4}, x_{2}x_{5}, x_{2}x_{6}, x_{2}x_{7}, x_{2}x_{8}, x_{2}x_{10}$ (7 edges)
\item $x_{3}: x_{3}x_{4}, x_{3}x_{5}, x_{3}x_{6}, x_{3}x_{7}$ (4 edges)
\item $x_{4}: x_{4}x_{5}, x_{4}x_{6}, x_{4}x_{7}, x_{4}x_{8}$ (4 edges)
\item $x_{5}: x_{5}x_{7}, x_{5}x_{8}, x_{5}x_{10}$ (3 edges)

\item $x_{7}: x_{7}x_{8}$ (1 edge)
\item $x_{8}: x_{8}x_{10}$ (1 edge)
\end{itemize}

Consider the red edge $x_{2}x_{5}$, which has $9$ needs. The g-b need is met by $x_{0}$, and the g-g need is met by $x_1$. Now the r-r need is met by $x_3, x_4, x_7, x_{8}, x_{10}$. We can meet the r-g need with $x_6$ and an arbitrary need with $x_{9}$. This leaves $4$ unsatisfied needs, necessitating $4$ additional points. So we require at least $15$ points.
\end{proof}

We now use a SAT solver to improve our lower bound. For $n \geq 15$, we build a Boolean formula $\Phi^{(n)}$ whose satisfiability is a necessary condition for $33_{65}$ to be representable on $n$ points. So if $\Phi^{(n)}$ is not satisfiable, then $33_{65}$ does not admit a representation on $n$ points.

For all $0 \leq i, j < n$ and all $0 \leq k \leq 2$, we have a variable $\phi_{i,j,k}$. We interpret $k = 0$ to be the color red, $k = 1$ to be the color blue, and $k = 2$ to be the color green. So $\phi_{i,j,k} = \texttt{true}$ indicates that the edge $x_{i}x_{j}$ is colored using $k$. We now define a formula to assert that each edge $x_{i}x_{j}$ receives exactly one color:
\[
\Phi_{0}^{(n)} := \bigwedge_{i<j} [ (\phi_{i,j,0} \vee \phi_{i,j,1} \vee \phi_{i,j,2})  \wedge (\neg\phi_{i,j,0} \vee \neg\phi_{i,j,1})  \wedge (\neg\phi_{i,j,0} \vee \neg\phi_{i,j,2})  \wedge (\neg\phi_{i,j,1} \vee \neg\phi_{i,j,2}) ]
\]

Now consider the points in Figure \ref{fig:3365}. Let $R$ be the set of red edges pictured in Figure \ref{fig:3365} together with the red edges specified in \Lem{lem:LowerBound1}. Let $B$ be the set of blue edges appearing in Figure \ref{fig:3365}, and let $G$ be the set of green edges appearing in Figure \ref{fig:3365}.

Define:
\[
    \Phi_{1}^{(n)} = \left(\bigwedge_{(i,j)\in R} \phi_{i,j,0}\right) \wedge \left(\bigwedge_{(i,j)\in B} \phi_{i,j,1}\right) \wedge \left(\bigwedge_{(i,j)\in G} \phi_{i,j,2} \right)
\]

Then $\Phi_{1}^{(n)}$ asserts that any  edges colored in Figure \ref{fig:3365} and \Lem{lem:LowerBound1} are colored correctly. Define the sets:
\begin{align*}
\text{BN} := \{(0,0), (0, 1), (0, 2), (1,0), (2, 0), (1,1) \} \\
\text{GN} := \{ (0, 0), (0, 1), (0, 2), (1,0), (2, 0)\}
\end{align*}
to denote the blue needs and green needs respectively. We now define formulas to assert that each edge in Figure \ref{fig:3365} and \Lem{lem:LowerBound1} has its needs met.
\begin{align*}
\Phi_{2}^{(n)} &= \bigwedge_{(i,j)\in R}\left[ \bigwedge_{c_1,c_2\in \{0,1,2\}} \left(\bigvee_{i\neq k \neq j} \phi_{i,k,c_1}\wedge\phi_{k,j,c_2}\right)\right] \\
\Phi_{3}^{(n)} &= \bigwedge_{(i,j)\in B}\left[ \bigwedge_{(c_1,c_2)\in BN} \left(\bigvee_{i\neq k \neq j} \phi_{i,k,c_1}\wedge\phi_{k,j,c_2}\right)\right] \\
\Phi_{4}^{(n)} &= \bigwedge_{(i,j)\in G}\left[ \bigwedge_{(c_1,c_2)\in GN} \left(\bigvee_{i\neq k \neq j} \phi_{i,k,c_1}\wedge\phi_{k,j,c_2}\right)\right].
\end{align*}

We next define a formula to assert that every edge that is not pre-colored has its needs satisfied.

\begin{align*}
     \Phi_{5}^{(n)} = \bigwedge_{(i,j)\notin R\cup B \cup G } & \phi_{i,j,0} \wedge \left[ \bigwedge_{(c_1,c_2)\in \{0,1,2\}} \left(\bigvee_{i\neq k \neq j} \phi_{i,k,c_1}\wedge\phi_{k,j,c_2}\right)\right] \\
     \vee \ &\phi_{i,j,1} \wedge \left[ \bigwedge_{(c_1,c_2)\in BN} \left(\bigvee_{i\neq k \neq j} \phi_{i,k,c_1}\wedge\phi_{k,j,c_2}\right)\right] \\
      \vee \  &\phi_{i,j,2} \wedge \left[ \bigwedge_{(c_1,c_2)\in GN} \left(\bigvee_{i\neq k \neq j} \phi_{i,k,c_1}\wedge\phi_{k,j,c_2}\right)\right]. 
\end{align*}
   
It remains to forbid cycles that contain (i) at least one green edge, and (ii) only blue and green edges.
\begin{align*}
\Phi_{6}^{(n)} := \bigwedge_{i<j<k} [ &\neg(\varphi_{i,j,2} \land \bigwedge_{c_{1},c_{2} \in \{1,2\}} (\varphi_{i,k,c_{1}} \land \varphi_{j,k,c_{2}})) \land \\
&\neg(\varphi_{i,k,2} \land \bigwedge_{c_{1},c_{2} \in \{1, 2\}} (\varphi_{i,j,c_{1}} \land \varphi_{j,k,c_{2}})) \land \\
&\neg(\varphi_{j,k,2} \land \bigwedge_{c_{1},c_{2} \in \{1, 2\}} (\varphi_{i,j,c_{1}} \land \varphi_{i,j,c_{2}})) ]
\end{align*}

\begin{lemma}
If $15 \leq n \leq 23$, then $n \not \in \Spec(33_{65})$.
\end{lemma}

\begin{proof}
For each $15 \leq n \leq 23$, we define:
\[
\Phi^{(n)} := \bigwedge_{i=0}^{6} \Phi_{i}^{(n)}.
\]

\noindent We have verified the unsatisfiability of each such $\Phi^{(n)}$ with a SAT solver.
\end{proof}

\begin{remark}
Our SAT solver code can be found \href{https://github.com/michaellevet/Comer-Schemes/blob/main/33_65.ipynb}{here}.
\end{remark}

\subsection{$34_{65}$}

We again phrase our results in the language of graph theory (see \Rmk{rmk:Constraints}). We use the color red to correspond to the flexible atom, blue to correspond to the atom $b$, and green to correspond to the atom $c$. Thus, the forbidden triangles are the blue-blue-green and blue-green-green ones.

\begin{lemma}
Any finite representation of $34_{65}$ must have at least $13$ points.
\end{lemma}

\begin{proof}
We refer to Figure \ref{fig:3465}. Let $x_{0}x_{1}$ be a red edge. The needs of $x_{0}x_{1}$ are precisely $\{r, b, g \} \times \{r, b, g\}$, yielding a total of $11$ points. Now fix $i \in \{0, 1\}$. Suppose that $x_{i}x_{j}$ is blue and $x_{i}x_{k}$ is green (or vice-versa). As blue-blue-green and blue-green-green triangles are forbidden, $x_{j}x_{k}$ is necessarily red. In particular, this implies that the following edges are red:
\begin{itemize}
    \item $x_{2}$: $x_{2}x_{4}$, $x_{2}x_{6}$, $x_{2}x_{7}$, $x_{2}x_{8}$, $x_{2}x_{9}$ ($5$ edges)
    
    \item $x_{3}:$ $x_{3}x_{6}, x_{3}x_{7}, x_{3}x_{8}$ ($3$ edges)
    
    \item $x_{4}$: $x_{4}x_{5}, x_{4}x_{6}$, $x_{4}x_{7}, x_{4}x_{8}$ ($4$ edges)
    
    \item $x_{5}$: $x_{5}x_{7}, x_{5}x_{9}$ ($2$ edges)
    
    \item $x_{6}$: $x_{6}x_{7}, x_{6}x_{9}$ ($2$ edges)
\end{itemize}

Now take the red edge $x_{2}x_{7}$, which has $9$ needs. Observe that the blue-green need is witnessed twice by $x_{0}, x_{1}$, and the red-red need is witnessed twice by $x_{4}, x_{6}$.

We may maximize the number of needs satisfied with the following configuration.
\begin{itemize}
    \item $x_{3}$ witnesses the blue-red need.
    \item $x_{5}$ witnesses the green-red need.
    \item $x_{8}$ witnesses the red-blue need.
    \item $x_{9}$ witnesses the red-green need.
    \item $x_{10}$ witnesses the green-blue need.
\end{itemize}

So there are two unmet needs, which necessitate two additional points. This brings our total to a minimum of $13$ points.
\end{proof}

We use a SAT solver almost similarly as in the case of $33_{65}$ to improve the lower bound. Now consider the points in Figure \ref{fig:3365}. Let $R$ be the set of red edges pictured in Figure \ref{fig:3465} together with the red edges specified in \Lem{lem:LowerBound1}. Let $B$ be the set of blue edges appearing in Figure \ref{fig:3465}, and let $G$ be the set of green edges appearing in Figure \ref{fig:3465}.

Define:
\[
    \Phi_{1}^{(n)} = \left(\bigwedge_{(i,j)\in R} \phi_{i,j,0}\right) \wedge \left(\bigwedge_{(i,j)\in B} \phi_{i,j,1}\right) \wedge \left(\bigwedge_{(i,j)\in G} \phi_{i,j,2} \right)
\]

Then $\Phi_{1}^{(n)}$ asserts that any  edges colored in Figure \ref{fig:3365} and \Lem{lem:LowerBound1} are colored correctly. Define the sets:
\begin{align*}
\text{BN} := \{(0,0), (0, 1), (0, 2), (1,0), (2, 0), (1, 1) \} \\
\text{GN} := \{ (0, 0), (0, 1), (0, 2), (1,0), (2, 0), (2,2)\}
\end{align*}
to denote the blue needs and green needs respectively. We now define formulas to assert that each edge in Figure \ref{fig:3365} and \Lem{lem:LowerBound1} has its needs met.
\begin{align*}
\Phi_{2}^{(n)} &= \bigwedge_{(i,j)\in R}\left[ \bigwedge_{c_1,c_2\in \{0,1,2\}} \left(\bigvee_{i\neq k \neq j} \phi_{i,k,c_1}\wedge\phi_{k,j,c_2}\right)\right] \\
\Phi_{3}^{(n)} &= \bigwedge_{(i,j)\in B}\left[ \bigwedge_{(c_1,c_2)\in BN} \left(\bigvee_{i\neq k \neq j} \phi_{i,k,c_1}\wedge\phi_{k,j,c_2}\right)\right] \\
\Phi_{4}^{(n)} &= \bigwedge_{(i,j)\in G}\left[ \bigwedge_{(c_1,c_2)\in GN} \left(\bigvee_{i\neq k \neq j} \phi_{i,k,c_1}\wedge\phi_{k,j,c_2}\right)\right].
\end{align*}

We next define a formula to assert that every edge that is not pre-colored has its needs satisfied.

\begin{align*}
     \Phi_{5}^{(n)} = \bigwedge_{(i,j)\notin R\cup B \cup G } & \phi_{i,j,0} \wedge \left[ \bigwedge_{(c_1,c_2)\in \{0,1,2\}} \left(\bigvee_{i\neq k \neq j} \phi_{i,k,c_1}\wedge\phi_{k,j,c_2}\right)\right] \\
     \vee \ &\phi_{i,j,1} \wedge \left[ \bigwedge_{(c_1,c_2)\in BN} \left(\bigvee_{i\neq k \neq j} \phi_{i,k,c_1}\wedge\phi_{k,j,c_2}\right)\right] \\
      \vee \  &\phi_{i,j,2} \wedge \left[ \bigwedge_{(c_1,c_2)\in GN} \left(\bigvee_{i\neq k \neq j} \phi_{i,k,c_1}\wedge\phi_{k,j,c_2}\right)\right]. 
\end{align*}

We now define $\Phi_{6}^{(n)}$ to forbid triangles of the form blue-blue-green or blue-green-green.

\begin{align*}
\Phi_{6}^{(n)} := \bigwedge_{i<j<k} [ &\neg(\varphi_{i,j,1} \land \bigwedge_{\substack{c_{1},c_{2} \in \{1,2\}\\c_{1} \neq c_{2}}} (\varphi_{i,k,c_{1}} \land \varphi_{j,k,c_{2}})) \land \\
&\neg(\varphi_{i,k,1} \land \bigwedge_{\substack{c_{1},c_{2} \in \{1,2\}\\c_{1} \neq c_{2}}}(\varphi_{i,j,c_{1}} \land \varphi_{j,k,c_{2}})) \land \\
&\neg(\varphi_{j,k,1} \land \bigwedge_{\substack{c_{1},c_{2} \in \{1,2\}\\c_{1} \neq c_{2}}} (\varphi_{i,j,c_{1}} \land \varphi_{i,j,c_{2}})) ]
\end{align*}

\begin{theorem}
Let $13 \leq n \leq 23$. Then $n \not \in \Spec(34_{65}).$
\end{theorem}

\begin{proof}
For each $13 \leq n \leq 23$, define:
\[
\Phi^{(n)} := \bigwedge_{i=0}^{6} \Phi_{i}^{(n)}.
\]

\noindent We have verified the unsatisfiability of each such $\Phi^{(n)}$ with a SAT solver.
\end{proof}

\begin{remark}
Our SAT solver code can be found \href{https://github.com/michaellevet/Comer-Schemes/blob/main/34_65.ipynb}{here}.
\end{remark}

\section{Lower Bounds on Cyclic Group Representations of $33_{65}$}

In this section, we will establish the following:

\begin{theorem} \label{thm:3365CyclicGroup}
There is no group representation of $33_{65}$ over $(\mathbb{Z}/n\mathbb{Z}, +)$ for any $n \leq 120$.
\end{theorem}

To prove Theorem~\ref{thm:3365CyclicGroup}, we employ a SAT solver. Our code can be found \href{https://github.com/algorithmachine/FAC}{here}. We discuss how to build the Boolean formula here $\Phi^{(n)}$ that we ultimately feed to the SAT solver.

For each $i \in \mathbb{Z}/n\mathbb{Z}$ we have three variables: $x_{i,0}, x_{i,1}, x_{i,2}$. Here, we have that $x_{i,0} = \texttt{true}$ if and only if $i$ receives the color red (the flexible atom). Similarly, $x_{i,1} = \texttt{true}$ if and only if $i$ receives the color blue, and $x_{i,2} = \texttt{true}$ if and only if $i$ receives the color green. We first enforce that exactly one of $x_{i,0}, x_{i,1}, x_{i,2}$ is satisfiable:
\[
\Phi_{0}^{(n)} := \bigwedge_{i\in \mathbb{Z}/n\mathbb{Z}} \biggr[(x_{i,0} \vee x_{i,1} \vee x_{i,2}) \land \neg(x_{i,0} \land x_{i,1}) \land \neg(x_{i,0} \land x_{i,2}) \land \neg(x_{i,1} \land x_{i,2}) \biggr].
\]

\noindent As $33_{65}$ is a symmetric relation algebra, we must enforce that the group representation is symmetric. Precisely, we must enforce that $i$ and $-i$ receive the same color (we are considering the additive group $\mathbb{Z}/n\mathbb{Z}$, and so $-i$ is the additive inverse of $i$). We enforce symmetry with the following constraint:
\[
\Phi_{1}^{(n)} := \bigwedge_{i \in \mathbb{Z}/n\mathbb{Z}} (x_{i,0} = x_{-i,0}) \land (x_{i,1} = x_{-i,1}) \land (x_{i,2} = x_{-i,2}).
\]

\noindent We next enforce that each edge has its needs met. Denote:
\begin{align*}
&\text{RN} := \{ 0,1,2\} \times \{0,1,2\}, \\
&\text{BN} := \{(0,0), (0,1), (1,0), (0,2), (2,0), (1,1)\}, \\
&\text{GN} := \{(0,0), (0,1), (0,2), (1,0), (2,0)\}
\end{align*}

\noindent to be the sets of needs for the red, blue, and green edges respectively. Define:
\begin{align*}
\Phi_{2}^{(n)} := \bigwedge_{i \in \mathbb{Z}/n\mathbb{Z}} \biggr[&\biggr((x_{i,0} \land \bigwedge_{(u,v) \in \text{RN}} \bigvee_{\substack{y \in \mathbb{Z}/n\mathbb{Z} \\ y \neq i}} x_{y,u} \land y_{(i-y) \text{ mod } n,v} \biggr) \\
&\biggr((x_{i,1} \land \bigwedge_{(u,v) \in \text{BN}} \bigvee_{\substack{y \in \mathbb{Z}/n\mathbb{Z} \\ y \neq i}} x_{y,u} \land y_{(i-y) \text{ mod } n,v} \biggr) \\
&\biggr((x_{i,2} \land \bigwedge_{(u,v) \in \text{RN}} \bigvee_{\substack{y \in \mathbb{Z}/n\mathbb{Z} \\ y \neq i}} x_{y,u} \land y_{(i-y) \text{ mod } n,v} \biggr)
\biggr]
\end{align*}

\noindent Here, $\Phi_{2}^{(n)}$ enforces that each edge has all of its needs met.

Finally, we will define a Boolean formula to enforce that there are no forbidden cycles. Recall for $33_{65}$ that the forbidden cycles are of the form $ccc$ (green, green,green), $bcc$ (blue, green, green), and $cbb$ (green, blue, blue). To accomplish this, we explicitly enumerate the permissible cycles: any cycle containing the flexible atom $a$ (red), and the all-blue triangle $bbb$. Define: 
\[
\Phi_{3}^{(n)} := \bigwedge_{\substack{i,j \in \mathbb{Z}/n\mathbb{Z} \\ i+j \neq 0 \text{ mod } n}} \neg(x_{i,0} \vee x_{j,0} \vee x_{i+j,0}) \implies (x_{i,1} \land x_{j,1} \land x_{i+j, 1}).
\]

\noindent Here, we are enforcing that if none of $i, j, i+j \pmod{n}$ are red (flexible), then $(i, j, i+j)$ form an all-blue triangle. Now define:
\[
\Phi^{(n)} = \Phi_{0}^{(n)} \land \Phi_{1}^{(n)} \land \Phi_{2}^{(n)} \land \Phi_{3}^{(n)}.
\]

\section{Conclusion}

We extended the notion of a Ramsey scheme by relaxing the condition that the cosets need to be sum-free. Using this combinatorial construction, we investigated the integral symmetric relation algebras $\mathcal{A}_{n}([i, i+j, i+\ell])$ that forbid the cycle configurations $\{ a_{i}a_{i+j}a_{i+\ell} : 0 \leq i < n\}$. As an application, we showed that $34_{65}$ admits a finite representation over the Comer scheme $C(24, 3697)$.

We also established several lower bounds. We showed that the minimum element in \linebreak $\Spec(\mathcal{A}_{3}([i,i,i]))$ is at least $13$ (folklore -- see Remark~\ref{rmk:LowerBounds}), and the minimum element in \linebreak$\Spec(\mathcal{A}_{4}([i,i,i]))$ is at least $21$. Furthermore, we showed that the minimum element in both $\Spec(33_{65})$ and $\Spec(34_{65})$ is at least $24$.

We conclude with several open problems.

\begin{conjecture}[Strong Commutative Flexible Atom Conjecture]
Every finite \textit{commutative} integral RA with a flexible atom is representable over a Comer scheme.
\end{conjecture}

\begin{remark}
The assumption that the relation algebras are commutative is indeed  necessary, and we discovered this only after the conference version \cite{AALRamics23} where we omitted the \textit{commutative} hypothesis. The algebra $1284_{1316}$ has diversity atoms $a$, $b$, $r$, and $\breve{r}$, where $a$ is flexible and the only  forbidden diversity cycle is $rbr$. However, $1284_{1316}$ is not commutative, and therefore not representable over an Abelian group. In particular, Comer's method does not yield a group representation.

It would even be of interest to look at the special case of symmetric relation algebras. Note that if all our diversity atoms are symmetric, then the relation algebra is commutative. However, the first and third authors had considerable success leveraging Comer's method to provide finite representations for non-symmetric, but commutative relation algebras \cite{AlmLevet2}.
\end{remark}

This paper contains some evidence for this conjecture, as does \cite{AlmLevet2}, which contains many new finite representations over Comer schemes. The infinite families of Directed Anti-Ramsey algebras from \cite{AlmLevet2}, and the infinite family $\{ \mathcal{A}_n([i,i,i+1]) : n > 4\}$, both appear to be ``mostly'' representable, based on the evidence. Except for some exceptions for small $n$, Comer schemes seem to yield representations.  The quasirandom nature of the sum-product interaction in finite fields (see \cite{Green} for example) seems to suggest a heuristic similar to one used in number theory: \emph{any potential structure in the primes not ruled out by obvious considerations can probably be found}. Case in point: we looked for an $n$ and a $p$ such that $\mathcal{A}_n([i,i,i+n/2])$ was representable over $\mathbb{F}_p$. And we found one.  There are likely infinitely many such $n$ and $p$, and $34_{65}$ will embed in all of them. 

There are (at least) two ingredients needed to complete the proof: (i) prove that Comer schemes are actually quasirandom in a relevant sense; and (ii) prove that for every RA with a flexible atom, there is an algebra (for instance,  $\mathcal{A}_n([i,i+j,i+\ell])$, but presumably more general) into which it embeds that admits a representation over a Comer scheme.

\begin{problem}
Formulate a suitable notion of quasirandomness for sequences of relation algebra atom structures or Comer schemes.
\end{problem}

The relation algebra $33_{65}$ is not known to admit a finite representation. In fact, it is the last remaining algebra in the family $N_{65}$ that has a flexible atom and for which no finite representation is known.

\begin{problem}
Find a forbidden scheme that would admit a representation of $33_{65}$. 
\end{problem}

\begin{problem}
The third author conjectures that there may exist a group representation of $33_{65}$ over a symmetric group. As the minimum element of $\Spec(33_{65})$ is at least $24$, we ask whether $33_{65}$ admits a group representation over $S_{4}$. Similarly, we ask whether $33_{65}$ admits a group representation over $S_{5}$.
\end{problem}

\begin{problem}
Algebra $\mathcal{A}_3([i,i,i+1])$ is $42_{65}$ and is not representable. (This fact is folklore and is probably not published. The reason for non-representability is failure to satisfy the equations Maddux calls (J), (L), and (M). See \cite[Page~30]{Madd} or \cite[Prop.~1]{KraMadd}.) Algebra  $\mathcal{A}_4([i,i,i+1])$ is  not representable via Comer's method (see Table \ref{tab:my_label}). Is it representable by some other method?
\end{problem}

\begin{problem}
Algebra $\mathcal{A}_6([i,i,i+2])$ is  not representable via Comer's method (see Table \ref{tab:my_label}).  Is it representable by some other method?
\end{problem}

\begin{problem}
Algebra $\mathcal{A}_5([i,i,i+1])$ is representable over $\mathbb{F}_{61}$. Is it representable over an infinite set? Over $\mathbb{Z}$?
\end{problem}

\appendix
\section{Point Algebra is not Finitely Representable} \label{appendix:PointAlgebra}

The \textit{point algebra} is the relation algebra with atoms $1', r, \breve{r}$ and the forbidden cycle $rr\breve{r}$. It is well-known amongst relation algebra specialists that the point algebra admits a representation over an infinite set, but is not finitely representable. We were initially unable to find a reference for this, and so we included a proof here for completeness. We thank an anonymous referee for pointing out that that this also follows from existing works \cite{semrl2021finite, NeuzerlingThesis, Neuzerling2015UndecidabilityOR}.

\begin{proposition}
The point algebra is representable over an infinite set, but not a finite set.
\end{proposition}

\begin{proof}
To see that the point algebra is representable,  Let $U = \mathbb{Q}$, and interpret to atom $r$ as the less-than relation on $\mathbb{Q}$, i.e., $\{(x,y) : x,y \in \mathbb{Q}, x<y \}$, and $\breve{r}$ as the greater-than relation.  

To see that no representation over a finite set is possible, we observe first that the interpretation of $r$ has to be a strict total ordering on $U$. Let $(u,v) \in r$. (We say $u$ is less than $v$.)  Because $r$ is contained in $\breve{r}\circ r $, there is some point $w$ less than both $u$ and $v$. Similarly, because $r$ is contained in $r\circ \breve{r}$, there is some point $x$ greater than both $u$ and $v$. Hence the ordering has no endpoints, so cannot be finite. 
\end{proof}

\bibliographystyle{alphaurl}
\bibliography{references}

\newpage

\begin{figure} 
\begin{center}
\begin {tikzpicture}[semithick, node distance =1 cm and 4cm]

	\node (C) {$x_{2}$};
	\node (A) [below left = of C] {$x_{0}$};
	\node (B) [below right = of C] {$x_{1}$};
	\node (D) [above = of C] {$x_{3}$};
	\node (E) [above = of D] {$x_{4}$};
	\node (F) [above = of E] {$x_{5}$};
	\node (G) [above = of F] {$x_{6}$};
	\node (H) [above = of G] {$x_{7}$};
	\node (I) [above = of H] {$x_{8}$};
	\node (J) [above = of I] {$x_{9}$};
	\node (K) [above = of J] {$x_{10}$};

	\path (A) edge[processred] node[above] {} (B);
	\path (A) edge[ForestGreen] node[above] {} (C);
	\path (B) edge[ForestGreen] node[left] {} (C);
	\path (A) edge[ForestGreen] node[above] {} (D);
	\path (B) edge[processred] node[left] {} (D);
	\path (A) edge[ForestGreen] node[above] {} (E);
	\path (B) edge[processblue] node[left] {} (E);
 
	\path (A) edge[processblue] node[above] {} (F);
	\path (B) edge[ForestGreen] node[left] {} (F);
	\path (A) edge[processblue] node[above] {} (G);
	\path (B) edge[processred] node[left] {} (G);
	\path (A) edge[processblue] node[above] {} (H);
	\path (B) edge[processblue] node[left] {} (H);
	\path (A) edge[processred] node[above] {} (I);
	\path (B) edge[ForestGreen] node[left] {} (I);
	\path (A) edge[processred] node[above] {} (J);
	\path (B) edge[processred] node[left] {} (J);
	\path (A) edge[processred] node[above] {} (K);
	\path (B) edge[processblue] node[left] {} (K);
	
	\path (C) edge[processred] node{} (D);
	\path (C) edge[processred, bend left =25] node{} (E);
	\path (C) edge[processred, bend right =25] node{} (F);
	\path (C) edge[processred, bend left =35] node{} (G);
	\path (C) edge[processred, bend right =35] node{} (H);
	\path (C) edge[processred, bend left =40] node{} (I);
	\path (C) edge[processred, bend right =40] node{} (K);
	
	\path (D) edge[processred] node{} (E);
	\path (D) edge[processred, bend left =25] node{} (F);
	\path (D) edge[processred, bend right =25] node{} (G);
	\path (D) edge[processred, bend left =25] node{} (H);
	
	\path (E) edge[processred] node{} (F);
	\path (E) edge[processred, bend left = 25] node{} (G);
	\path (E) edge[processred, bend right = 25] node{} (H);
	\path (E) edge[processred, bend left =25] node{} (I);
	
	\path (F) edge[processred, bend left = 25] node{} (H);
	\path (F) edge[processred, bend right = 25] node{} (I);
	\path (F) edge[processred, bend left =25] node{} (K);
	
	\path (H) edge[processred] node{} (I);
	\path (I) edge[processred, bend left =25] node{} (K);
	\end{tikzpicture}  
\end{center}
\caption{Subgraph which must appear off any red edge in a representation of $33_{65}$.}
\label{fig:3365}
\end{figure}

\begin{figure} 
\begin{center}
\begin {tikzpicture}[semithick, node distance =1 cm and 4cm] 

	\node (C) {$x_{2}$};
	\node (A) [below left = of C] {$x_{0}$};
	\node (B) [below right = of C] {$x_{1}$};
	\node (D) [above = of C] {$x_{3}$};
	\node (E) [above = of D] {$x_{4}$};
	\node (F) [above = of E] {$x_{5}$};
	\node (G) [above = of F] {$x_{6}$};
	\node (H) [above = of G] {$x_{7}$};
	\node (I) [above = of H] {$x_{8}$};
	\node (J) [above = of I] {$x_{9}$};
	\node (K) [above = of J] {$x_{10}$};

	\path (A) edge[processred] node[above] {} (B);
	\path (A) edge[processblue] node[above] {} (C);
	\path (B) edge[processblue] node[left] {} (C);
	\path (A) edge[processblue] node[above] {} (D);
	\path (B) edge[processred] node[left] {} (D);
	\path (A) edge[processblue] node[above] {} (E);
	\path (B) edge[ForestGreen] node[left] {} (E);
	\path (A) edge[processred] node[above] {} (F);
	\path (B) edge[processblue] node[left] {} (F);
	\path (A) edge[ForestGreen] node[above] {} (G);
	\path (B) edge[processblue] node[left] {} (G);
	\path (A) edge[ForestGreen] node[above] {} (H);
	\path (B) edge[ForestGreen] node[left] {} (H);
	\path (A) edge[ForestGreen] node[above] {} (I);
	\path (B) edge[processred] node[left] {} (I);
	\path (A) edge[processred] node[above] {} (J);
	\path (B) edge[ForestGreen] node[left] {} (J);
	\path (A) edge[processred] node[above] {} (K);
	\path (B) edge[processred] node[left] {} (K);
	
	\path (C) edge[processred, bend left = 25] node{} (E);
	\path (C) edge[processred, bend right = 25] node{} (G);
	\path (C) edge[processred, bend left = 30] node{} (H);
	\path (C) edge[processred, bend right = 30] node{} (I);
	\path (C) edge[processred, bend left = 35] node{} (J);
	
	\path (D) edge[processred, bend left = 25] node{} (G);
	\path (D) edge[processred, bend right = 25] node{} (H);
	\path (D) edge[processred, bend left = 35] node{} (I);
	
	\path(E) edge[processred] node{} (F);
	\path (E) edge[processred, bend left = 25] node{} (G);
	\path (E) edge[processred, bend right = 25] node{} (H);
	\path (E) edge[processred, bend left = 35] node{} (I);
	
	\path (F) edge[processred, bend right = 25] node{} (H);
	\path (F) edge[processred, bend left = 35] node{} (J);
	
	\path (G) edge[processred, bend right = 25] node{} (H);
	\path (G) edge[processred, bend left = 35] node{} (J);
	\end{tikzpicture}  
\end{center}
\caption{Subgraph which must appear off any red edge in a representation of $34_{65}$.}
\label{fig:3465}
\end{figure}

\end{document}